\documentclass[11pt]{amsart}
\usepackage{a4wide,color}
\usepackage[bookmarks=true,hyperindex,pdftex,colorlinks,citecolor=blue,linkcolor=blue]{hyperref}

\newtheorem{theorem}{Theorem}[section]
\newtheorem{problem}[theorem]{Problem}
\newtheorem{corollary}[theorem]{Corollary}

\newtheorem{lemma}[theorem]{Lemma}

\theoremstyle{definition}
\newtheorem{definition}[theorem]{Definition}
\newtheorem{remark}[theorem]{Remark}

\newcommand{\w}{\omega}
\newcommand{\e}{\varepsilon}
\newcommand{\IF}{\mathbb F}
\newcommand{\IN}{\mathbb N}
\newcommand{\IR}{\mathbb R}

\title{Banach actions preserving unconditional convergence}
\author{Taras Banakh and Vladimir Kadets}
\address{Ivan Franko National University of Lviv (Ukraine) and Jan Kochanowski University in Kielce (Poland); \href{http://orcid.org/0000-0001-6710-4611}{ORCID: \texttt{0000-0001-6710-4611} }}
\email{t.o.banakh@gmail.com, taras.banakh@lnu.edu.ua, taras.banakh@ujk.edu.pl}
\address{V.N.Karazin Kharkiv National University,
4 Svobody sq., Kharkiv, 61022, Ukraine \newline
\href{http://orcid.org/0000-0002-5606-2679}{ORCID: \texttt{0000-0002-5606-2679} }}
\email{v.kateds@karazin.ua}

\begin{document}

\begin{abstract}Let $A,X,Y$ be Banach spaces and $A\times X\to Y$, $(a,x)\mapsto ax$, be a continuous bilinear function, called a {\em Banach action}. We say that  this action {\em preserves unconditional convergence} if for every bounded sequence $(a_n)_{n\in\w}$ in $A$ and unconditionally convergent series $\sum_{n\in\w}x_n$ in $X$ the series $\sum_{n\in\w}a_nx_n$ is unconditionally convergent in $Y$. We prove that a Banach action $A\times X\to Y$ preserves unconditional convergence if and only if for any linear functional $y^*\in Y^*$ the operator $D_{y^*}:X\to A^*$, $D_{y^*}(x)(a)=y^*(ax)$, is absolutely summing. Combining this characterization with the famous Grothendieck theorem on the absolute summability of operators from $\ell_1$ to $\ell_2$, we prove that a Banach action $A\times X\to Y$ preserves unconditional convergence if $A$ is a Hilbert space possessing an orthonormal basis $(e_n)_{n\in\w}$ such that for every $x\in X$ the series $\sum_{n\in\w}e_nx$ is weakly absolutely convergent. Applying known results of Garling on the absolute summability of diagonal operators between sequence spaces, we prove that for (finite or infinite) numbers $p,q,r\in[1,\infty]$ with $\frac1r\le\frac1p+\frac1q$, the coordinatewise multplication $\ell_p\times\ell_q\to\ell_r$ preserves unconditional convergence if and only if one of the following conditions holds: (i) $p\le 2$ and $q\le r$, (ii) $2<p<q\le r$, (iii) $2<p=q<r$, (iv) $r=\infty$, (v) $2\le q<p\le r$, (vi) $q<2<p$ and $\frac1p+\frac1q\ge\frac1r+\frac12$.
\end{abstract}
\maketitle

\section{Introduction}

By a {\em Banach action} we understand any continuous bilinear function $A\times X\to Y$, $(a,x)\mapsto ax$, defined on the product $A\times X$ of Banach spaces $A,X$ with values in a Banach space $Y$. The Banach space $A$ is called the {\em acting space} of the action $A\times X\to Y$. 

We say that a Banach action $A\times X\to Y$ {\em preserves unconditional convergence} if for any unconditionally convergent series $\sum_{n\in\w}x_n$ in $X$ and any bounded sequence $(a_n)_{n\in \w}$ in $A$ the series $\sum_{n\in\w}a_nx_n$ converges unconditionally in the Banach space $Y$. Let us recall \cite[1.c.1]{LT} that a series $\sum_{n\in\w}x_n$ in a Banach space $X$ {\em converges unconditionally} if for any permutation $\sigma$ of $\w=\{0,1,2,\dots\}$ the series $\sum_{n\in\w}x_{\sigma(n)}$ converges in $X$.

 Observe that the operation of multiplication $X\times X\to X$, $(x,y)\mapsto xy$, in a Banach algebra $X$ is a Banach action. The problem of recognition of Banach algebras whose multplication preserves unconditional convergence has been considered in the paper \cite{BR}, which motivated us  to explore the following general question. 
 
\begin{problem}\label{prob1} Given a Banach action, recognize whether it  preserves unconditional convergence.
\end{problem}

This problem is not trivial even for the Banach action $\ell_p\times\ell_q\to\ell_r$ assigning to every pair $(x,y)\in\ell_p\times\ell_q$ their coordinatewise product $xy\in\ell_r$. The classical H\"older inequality implies that the coordinatewise multiplication $\ell_p\times\ell_q\to\ell_r$ is well-defined and continuous for any (finite or infinite) numbers $p,q,r\in[1,\infty]$ satisfying the inequality $\frac1r\le\frac1p+\frac1q$. 

Let us recall that $\ell_p$ is the Banach space of all sequences $x:\w\to\IF$ with values in the field $\IF$ of real or complex numbers such that $\|x\|_{\ell_p}<\infty$ where 
$$
\|x\|_{\ell_p}=
\begin{cases}
\big(\sum_{n\in\w}|x(n)|^p\big)^{\frac1p}<\infty&\mbox{if $p\in[1,\infty)$};\\
\phantom{\,}\sup_{n\in \w}|x(n)|&\mbox{if $p=\infty$.}
\end{cases}
$$

One of the main results of this paper is the following theorem answering Problem~\ref{prob1} for the Banach actions $\ell_p\times\ell_q\to\ell_r$.

\begin{theorem}\label{t:main2} For numbers $p,q,r\in[1,\infty]$ with $\frac1{r}\le\frac{1}p+\frac{1}q$, the coordinatewise multiplication $\ell_p\times\ell_q\to\ell_r$ preserves unconditional convergence if and and only if one of the following conditions is satisfied:
\begin{itemize}
\item[{\rm(i)}] $p\le 2$ and $q\le r$.
\item[{\rm(ii)}] $2<p<q\le r$; 
\item[{\rm(iii)}] $2<p=q<r$;
\item[{\rm(iv)}] $p=q=r=\infty$;
\item[{\rm(v)}] $2\le q<p\le r$;
\item[{\rm(vi)}] $q<2<p$ and $\frac1p+\frac1q\ge\frac1r+\frac12$.
\end{itemize}
\end{theorem}

Theorem~\ref{t:main2} implies the following characterization whose ``only if'' part is due to Daniel Pellegrino (private communication), who proved it using the results of Bennett \cite{Bennett}.

\begin{corollary}\label{c:main} For a number $p\in[1,\infty]$,  the coordinatewise multiplication $\ell_p\times\ell_p\to\ell_p$ preserves unconditional convergence if and only if $p\in[1,2]\cup\{\infty\}$.
\end{corollary}

The other principal result of the paper is the following partial answer to Problem~\ref{prob1}.

\begin{theorem}\label{t:main1} A Banach action $A\times X\to Y$ preserves unconditional convergence if $A$ is a Hilbert space possessing an orthonormal basis $(e_n)_{n\in\w}$  such that for every $x\in X$ the series $\sum_{n\in\w}e_nx$ is unconditionally convergent in $Y$. 
\end{theorem}

Theorems \ref{t:main2} and \ref{t:main1} will be proved in Sections~\ref{s:main2} and \ref{s:main1}, respectively. In Section~\ref{s:U} we shall prove two characterizations of Banach actions that preserve unconditional convergence.  One of these characterizations (Theorem~\ref{t:dual}) reduces the problem of recognizing Banach actions preserving unconditional convergence to the problem of recognizing absolutely summing operators, which is well-studied in Functional Analysis, see \cite{DJT}, \cite{Woj}.

\begin{remark} It should be mentioned that problems similar to Problem~\ref{prob1} have been considered in the mathematical literature. In particular,  Boyko \cite{Boyko} considered a problem of recognizing subsets $G$ of the Banach space  $L(X,Y)$ of continuous linear operators from a Banach space $X$ to a Banach space $Y$ such that for any unconditionally convergent series $\sum_{i\in\w}x_i$ in $X$ and any sequence of operators $\{T_n\}_{n\in\w}\subseteq G$ the series $\sum_{n\in\w}T_n(x)$ converges (unconditionally or absolutely) in $Y$. 
\end{remark}

\section{Preliminaries}

Banach spaces considered in this paper are over the field $\IF$ of real or complex numbers. For a Banach space $X$ its norm is denoted by $\|\cdot\|_X$ or $\|\cdot\|$ (if $X$ is clear from the context). The dual Banach space to a Banach space $X$ is denoted by $X^*$.

 By $\w$ we denote the set of all non-negative integer numbers.  Each number $n\in\w$ is identified with the set $\{0,\dots,n-1\}$ of smaller numbers. Let $\IN=\w\setminus\{0\}$ be the set of positive integer numbers. For a set $A$ let $[A]^{<\w}$ denote the family of all finite subsets of $A$.

We start with two known elementary lemmas, giving their proofs just for the reader's  convenience.

\begin{lemma}\label{l:R} For any finite sequence of real numbers $(x_k)_{k\in n}$ we have 
$$\sum_{k\in n}|x_k|\le 2\max_{F\subseteq n}\big|\sum_{k\in F}x_k\big|.
$$
\end{lemma}

\begin{proof} Let $F_+=\{k\in n:x_k\ge 0\}$ and $F_-=\{k\in n:x_k<0\}$.
Then $$\sum_{k\in n}|x_k|=\big|\sum_{k\in F_+}x_k\big|+\big|\sum_{k\in F_-}x_k\big|\le 2\max_{F\subseteq n}\big|\sum_{k\in n}z_k\big|.$$
\end{proof}

\begin{lemma}\label{l:Z} For any finite sequence of complex numbers $(z_k)_{k\in n}$ we have 
$$\sum_{k\in n}|z_k|\le 4\max_{F\subseteq n}\big|\sum_{k\in F}z_k\big|.
$$
\end{lemma}

\begin{proof} For a complex number $z$, let $\Re(z)$ and $\Im(z)$ be its real and complex parts, respectively. Applying Lemma~\ref{l:R} we conclude that 
\begin{multline*}
\sum_{k\in n}|z_k|\le\sum_{k\in n}(|\Re(z_k)|+|\Im(z_k)|)=\sum_{k\in n}|\Re(z_k)|+\sum_{k\in n}|\Im(z_k)|\le\\
2\max_{F\subseteq n}\big|\sum_{k\in F}\Re(z_k)\big|+
2\max_{F\subseteq n}\big|\sum_{k\in F}\Im(z_k)\big|=\\
2\max_{F\subseteq n}\big|\Re\big(\sum_{k\in F}z_k\big)\big|+
2\max_{F\subseteq n}\big|\Im\big(\sum_{k\in F}z_k\big)\big|\le4\max_{F\subseteq n}\big|\sum_{k\in F}z_k\big|.
\end{multline*}
\end{proof}

\begin{remark} It is clear that the constant 2 in Lemma~\ref{l:R} is the best possible. On the other hand, the constant 4 in Lemma~\ref{l:Z} can be improved to the constant $\pi$, which is the best possible according to \cite{NV}.
\end{remark}

The following inequality between $\ell_p$ and $\ell_q$ norms is well-known and follows from the H\"older inequality.

\begin{lemma}\label{l:p-q} For any $1\le p\le q<\infty$ and any sequence $(z_k)_{k\in n}$ of complex numbers we have
$$\Big(\sum_{k\in n}|z_k|^q\Big)^{\frac1q}\le \Big(\sum_{k\in n}|z_k|^p\Big)^{\frac1p}\le
n^{\frac1p-\frac1q}\Big(\sum_{k\in n}|z_k|^q\Big)^{\frac1q}.$$
\end{lemma}

By Proposition 1.c.1 in \cite{LT}, a series $\sum_{k\in\w}x_k$ in a Banach space $X$ converges unconditionally to an element $x\in X$ if and only if for any $\e>0$ there exists a finite set $F\subseteq\w$ such that $\big\|x-\sum_{k\in E}x_k\big\|<\e$ for any finite set $E\subseteq \w$ containing $F$. By Proposition 1.c.1 \cite{LT}, a series $\sum_{k\in \w}x_k$ in a Banach space $X$ converges unconditionally to some element of $X$ if and only if it is {\em unconditionally Cauchy} in the sense that for every $\e>0$ there exists a finite set $F\subset\w$ such that $\sup_{E\in[\w\setminus F]^{<\w}}\big\|\sum_{k\in E}x_k\big\|<\e$. 

By the Bounded Multiplier Test \cite[1.6]{DJT}, a series $\sum_{k\in\w}x_k$ in a Banach space $X$ converges unconditionally if and only if for  every bounded sequence of scalars $(t_n)_{n\in\w}$ the series $\sum_{n=1}^{\infty}t_n x_{n}$ converges in $X$. This characterization suggests the possibility of replacing scalars $t_n$ by Banach action multipliers, which is the subject of our paper.

A series $\sum_{i\in \w}x_i$ in a Banach space $X$ is called {\em weakly absolutely convergent} if for every linear continuous functional $x^*$ on $X$ we have $\sum_{n\in\w}|x^*(x_n)|<\infty$. It is easy to see that each unconditionally convergent series in a Banach space is weakly absolutely convergent. By Bessaga--Pe\l cz\'nski Theorem \cite[6.4.3]{KK}, the converse is true if and only if the Banach space $X$ contains no subspaces isomorphic to $c_0$.

For a Banach space $X$, let $\Sigma[X]$ be the Banach space of all functions $x:\w\to X$ such that the series $\sum_{n\in\w}x(n)$ is unconditionally  Cauchy. The space $\Sigma[X]$ is endowed with the norm
$$\|x\|=\sup_{F\in[\w]^{<\w}}\big\|\sum_{n\in F}x(n)\big\|_X.$$
The space $\Sigma[X]$ is called {\em the Banach space of unconditionally convergent series} in the Banach space $X$.

 More information on unconditional convergence in Banach spaces can be found in the monographs \cite{Diestel}, \cite{DJT}, \cite{KK}, \cite{LT}, \cite{Woj}.

\begin{lemma}\label{l:1} Let $X,Y$ be Banach spaces and $(T_n)_{n\in\w}$ be a sequence of bounded operators from $X$ to $Y$ such that for every $x\in X$ the series $\sum_{n\in\w}T_n(x)$ converges unconditionally in $Y$. Then there exists a real constant $C$ such that 
$$\sup_{F\in[\w]^{<\w}}\big\|\sum_{n\in F}T_n(x)\big\|\le C\|x\|.$$
\end{lemma}

\begin{proof} The sequence $(T_n)_{n\in\w}$ determines a linear operator $T:X\to\Sigma[Y]$ whose graph
$$\bigcap_{n\in\w}\big\{(x,y)\in X\times\Sigma[Y]:y(n)=T_n(x)\big\}$$is closed in the Banach space $X\times\Sigma[Y]$. 
By the Closed Graph Theorem, the operator $T$ is bounded and hence
$$
\sup_{F\in[\w]^{<\w}}\big\|\sum_{n\in F}T_n(x)\big\|=\|T(x)\|_{\Sigma[Y]}\le\|T\|\cdot\|x\|.$$
\end{proof}

We shall often use the following Closed Graph Theorem for multilinear operators proved by Fernandez in \cite{Fernandez}.

\begin{theorem}\label{t:Fernandez} A multilinear operator $T:X_1\times \cdots \times X_n\to Y$ between Banach spaces is continuous if and only if it has closed graph if and only if it has bounded norm
$$\|T\|=\sup\{\|T(x_1,\dots,x_n)\|:\max\{\|x_1\|,\dots,\|x_n\|\}\le 1\}.$$
\end{theorem}

\section{Characterizing Banach actions that preserve unconditional convergence}\label{s:U}

In this section we present two characterizations of Banach actions that preserve unconditional convergence.

\begin{definition}\label{d:unconditional} A Banach {action} $A\times X\to Y$ is called {\em unconditional} if there exists a positive real number $C$ such that for every $n\in\IN$ and sequences $\{a_k\}_{k\in n}\subset A$ and $\{x_k\}_{k\in n}\subset X$ we have
$$\big\|\sum_{k\in n}a_kx_k\big\|_Y\le C\cdot \max_{k\in n}\|a_k\|_A\cdot\max_{F\subseteq n}\big\|\sum_{k\in F}x_k\big\|_X.$$
\end{definition}

\begin{theorem}\label{t:unconditional} A Banach {action} $A\times X\to Y$  preserves unconditional convergence if and only if it is unconditional.
\end{theorem}

\begin{proof} To prove the ``if'' part, assume that the {action} $A\times X\to Y$ is unconditional and hence satisfies Definition~\ref{d:unconditional} for some constant $C$. To prove that the {action} preserves unconditional convergence, fix any unconditionally convergent series $\sum_{n\in\w}x_n$ in $X$, a bounded sequence $(a_n)_{n\in\w}$ in $A$ and  $\e>0$. Let $a=\sup_{n\in\w}\|a_n\|_A<\infty$. By the unconditional convergence of the series $\sum_{n\in\w}x_n$, there exists a finite set $F\subset\w$ such that $$\sup_{E\in[\w\setminus F]^{<\w}}\big\|\sum_{n\in E}x_n\big\|_X<\frac{\e}{C(1+a)}.$$ Then for any finite set $E\subseteq\w\setminus F$ we have$$\big\|\sum_{n\in E}a_nx_n\big\|_Y\le C\cdot\max_{n\in E}\|a_n\|_A\cdot\max_{K\subseteq E}\big\|\sum_{n\in K}x_n\big\|_X\le Ca\frac{\e}{C(1+a)}\le\e,$$
which means that the series $\sum_{n\in\w}a_nx_n$ is unconditionally Cauchy and hence unconditionally convergent in the Banach space $Y$.
\smallskip

To prove the ``only if'' part, assume that a Banach multiplication $A\times X\to Y$ preserves unconditional convergence. Let $\Sigma[X]$ and $\Sigma[Y]$ be the Banach spaces of unconditionally convergent series in the Banach spaces $X$ and $Y$, respectively. 
 Let $\ell_\infty[A]$ be the Banach space of all bounded functions $a:\w\to A$ endowed with the norm $\|a\|_{\ell_\infty[A]}=\sup_{n\in\w}\|a(n)\|_A$. For every $a\in\ell_\infty[A]$ and $x\in\Sigma[X]$, consider the function $ax:\w\to Y$ assigning to each $n\in\w$ the element $a(n)x(n)\in Y$, which is the image of the pair $(a(n),x(n))$ under the Banach action $A\times X\to Y$. Since the action $A\times X\to Y$ preserves unconditional convergence, the function $ax$ belongs to the Banach space $\Sigma[Y]$ of all unconditionally convergent series on $Y$. Therefore, the Banach action    
$$T:\ell_\infty[A]\times\Sigma(X)\to\Sigma[Y],\quad T:(a,x)\mapsto ax,$$ is well-defined.
This {action} has closed graph
$$\bigcap_{n\in\w}\{((a,x),y)\in(\ell_\infty[A]\times\Sigma[X])\times\Sigma[Y]:y(n)=a(n)x(n)\}$$and hence is continuous, by Theorem~\ref{t:Fernandez}. 

Now take any $n\in \w$ and sequences $(a_k)_{k\in n}\in A^n$ and $(x_k)_{k\in n}\in X^n$. Consider the function $a:\w\to A$ defined by $a(k)=a_k$ for $k\in n$ and $a(k)=0$ for $k\in\w\setminus n$. Also let $x:\w\to X$ be the function such that $x(k)=x_k$ for $k\in n$ and $x(k)=0$ for $k\in\w\setminus n$. Since $a\in\ell_\infty[A]$ and $x\in\Sigma[X]$, we have
$$\big\|\sum_{k\in n}a_kx_k\big\|_Y \le \|ax\|_{\Sigma[Y]} \le \|T\|\cdot\|a\|_{\ell_\infty[A]}\cdot \|x\|_{\Sigma[X]}=\|T\|\cdot\max_{k\in n}\|a_k\|\cdot\max_{F\subseteq n}\big\|\sum_{k\in F}x_k\big\|_X,$$
which means that the Banach {action} $A\times X\to Y$ is unconditional.
\end{proof}

An essential ingredient of the proof of Theorems~\ref{t:main2} and \ref{t:main1} is the following characterization of unconditional Banach actions in terms of absolutely summing operators.  An operator $T:X\to Y$ between Banach spaces $X,Y$ is {\em absolutely summing} if for every unconditionally convergent series $\sum_{n\in\w}x_n$ in $X$ the series $\sum_{n\in\w}T(x_n)$ is absolutely convergent, i.e., $\sum_{n\in\w}\|T(x_n)\|<\infty$. For more information on absolutely summing operators, see  \cite{DJT} and \cite[Section III.F]{Woj}. 

Let $A,X,Y$ be Banach spaces over the field $\IF$ of real or complex numbers. Given a Banach action $A\times X\to Y$, consider the trilinear operator $$Y^*\times A\times X\to\IF,\quad (y^*,a,x)\mapsto y^*(ax),$$which induces the bilinear operator
$$D:Y^*\times X\to A^*,\quad D:(y^*,x)\mapsto D_{y^*\!,x},\quad\mbox{where}\quad D_{y^*\!,x}:a\mapsto y^*(ax).$$

For a Banach space $Y$, a subspace $E\subseteq Y^*$ is called {\em norming} if there exists a real constant $c$ such that $$\|y\|\le c\sup_{y^*\in S_E}|y^*(y)|,$$
where $S_E=\{e\in E:\|e\|=1\}$ is the unit sphere of the space $E$.

\begin{theorem}\label{t:dual} Let $Y$ be a Banach space and $E$ be a norming closed linear subspace in $Y^*$. A Banach action $A\times X\to Y$ is unconditional if and only if for every $y^*\in E$, the operator $D_{y^*}:X\to A^*$, $D_{y^*}:x\mapsto D_{y^*,x}$, is absolutely summing.
\end{theorem}

\begin{proof} Assuming that the action $A\times X\to Y$ is unconditional, find a real constant $C$ satisfying the inequality in Definition~\ref{d:unconditional}. 

Fix any $y^*\in E$, $n\in\IN$ and a sequence $(x_k)_{k\in n}$ of elements of the Banach space $X$. In the following formula by $S$ we shall denote the unit sphere of the Banach space $A$. For a sequence $a\in S^n$ and $k\in n$ by $a_k$ we denote the $k$-th coordinate of $a$. Applying Lemma~\ref{l:Z} and the inequality from Definition~\ref{d:unconditional}, we obtain that 
$$
\begin{aligned}
\sum_{k\in n}\|D_{y^*\!,x_k}\|&=\sum_{k\in n}\sup_{a\in S}|y^*(ax_k)|\le \sup_{a\in S^n}\sum_{k\in n}|y^*(a_kx_k)|\le\\
&\le4\sup_{a\in S^n}\max_{F\subseteq n}\big|\sum_{k\in F}y^*(a_kx_k)\big|\le
4\sup_{a\in S^n} \max_{F\subseteq n}\|y^*\|\cdot \big\|\sum_{k\in F}a_kx_k\big\|\le\\
&\le 4\sup_{a\in S^n} \max_{F\subseteq n}\|y^*\|\cdot C\max_{k\in F}\|a_k\|\cdot\max_{E\subseteq F}\big\|\sum_{k\in E}x_k\big\| \le 4C\|y^*\|\max_{E\subseteq n}\big\|\sum_{k\in E}x_k\big\|.
\end{aligned}
$$
This inequality implies that for every $y^*\in E$ and every unconditionally convergent series $\sum_{n\in\w}x_n$ in $X$ we have $\sum_{n\in \w}\|D_{y^*\!,x_n}\|<\infty$, which means that the operator $D_{y^*}:X\to A^*$, $D_{y^*}:x\mapsto D_{y^*,x}$, is absolutely summing.
\smallskip

Now assume conversely that for every  $y^*\in E$ the operator $D_{y^*}:X\to A^*$ is absolutely summing. Since the space $E$ is norming, there is a real constant $c$ such that $$\|y\|\le c\cdot\sup_{y^*\in S_E}|y^*(y)|$$ for every $y\in Y$. Let $\Sigma[X]$ be the Banach space of unconditionally convergent series in $X$ and $$\ell_1[A^*]=\{(a_n^*)_{n\in\w}\in (A^*)^\w:\sum_{n\in\w}\|a^*_n\|<\infty\}$$ be the Banach space of all absolutely summing sequences in $A^*$. The Banach space $\ell_1[A^*]$ is endowed with the norm $\|(a_n^*)_{n\in\w}\|=\sum_{n\in\w}\|a^*_n\|$. Our assumption ensures that the bilinear operator $$D_{\star\star}:E\times \Sigma[X]\to\ell_1[A^*],\quad D_{\star\star}:(y^*,(x_n)_{n\in\w})\mapsto(D_{y^*,x_n})_{n\in\w},$$is well-defined. It is easy to see that this operator has closed graph and hence it is continuous. 

Then for every $n\in\IN$ and sequences $\{a_k\}_{k\in n}\subset A$ and $\{x_k\}_{k\in n}\subset X$ we have
$$
\begin{aligned}
\big\|\sum_{k\in n}a_kx_k\big\|&
\le c\sup_{y^*\in S_E} \big|\sum_{k\in n}y^*(a_kx_k)\big|\le c\sup_{y^*\in S_E}\sum_{k\in n}|y^*(a_kx_k)|=c \sup_{y^*\in S_E}\sum_{k\in n}|D_{y^*,x_k}(a_k)| \le \\
&\le c\sup_{y^*\in S_E} \sum_{k\in n}\|D_{y^*,x_k}\|\cdot\|a_k\|\le c \sup_{y^*\in S_E}\max_{j\in n}\|a_j\|\sum_{k\in n}\|D_{y^*,x_k}\|\le\\
&\le c\sup_{y^*\in S_E} \max_{j\in n}\|a_j\|\cdot\|D_{\star\star}\|\cdot\|y^*\|\cdot\max_{F\subseteq n}\big\|\sum_{k\in F}x_k\big\|\le\\
&\le c \|D_{\star\star}\|\cdot\max_{j\in n}\|a_j\|\cdot\max_{F\subseteq n}\big\|\sum_{k\in F}x_k\big\|,
\end{aligned}
$$
which means that the Banach action $A\times X\to Y$ is unconditional.
\end{proof}

For any $p,q,r\in[1,\infty]$ with $\frac1r\le \frac1p+\frac1q$ and every $a \in \ell_p$ let $d_a:\ell_q\to\ell_r$, $d_a: x\mapsto a x$, be the (diagonal) operator of coordinatewise multiplication by $a$.

For a number $p\in[1,\infty]$, let $p^*$ be the unique number in $[1,\infty]$ such that $\frac1p+\frac1{p^{*}}=1$.  It is well-known that for any $p\in[1,\infty)$ the dual Banach space $\ell_p^*$ can be identified with $\ell_{p^*}$ and for $p = \infty$ a weaker condition holds true: $\ell_1$  is not equal to  $\ell_\infty^*$ but can be viewed as a norming subspace of $\ell_\infty^*$ (with norming constant $c=1$). 

Theorems~\ref{t:unconditional} and \ref{t:dual} imply the following characterization, that will be essentially used in the proof of Theorem~\ref{t:main2}.

\begin{corollary}\label{c:Kadets} For  numbers $p,q,r\in[1,\infty]$ with $\frac1r\le\frac1p+\frac1q$ the following conditions are equivalent:
\begin{enumerate}
\item the coordinatewise multiplication $\ell_p\times\ell_q\to\ell_r$ preserves unconditional convergence;
\item for every $a \in\ell_{r^{*}}$ the operator of coordinatewise multiplication $d_a:\ell_q\to\ell_{p^{*}}$, $d_a:x\mapsto a x$, is absolutely summing.
\end{enumerate}
\end{corollary}

Corollary~\ref{c:Kadets} motivates the problem of recognizing absolute summing operators among diagonal operators $d_a:\ell_{p^*}\to\ell_q$. This problem has been considered and resolved by Garling who proved the following characterization in \cite[Theorem 9]{Garling}. In this characterization, $\ell_{p^-}$ denotes the linear subspace of $\ell_p$ consisting of all sequences $x\in\ell_p$ such that
$$\sum_{n\in\w}\frac{|x(n)|^p}{1+\ln |a_n^{-1}|}<\infty.$$

\begin{theorem}[Garling]\label{t:Garling} For numbers $r,p,q\in[1,\infty]$ with $\frac1r+\frac1{p^*}\ge\frac1q$ and a sequence $a\in\ell_r$, the operator $d_a:\ell_{p^*}\to\ell_q$ is absolutely summing if and only if the following conditions are satisfied:
\begin{itemize}
\item[(i)] if $1\le p\le 2$ and $p<q$, then $a\in\ell_p$;
\item[(ii)] if $1<p=q<2$, then $a\in \ell_{p^-}$;
\item[(iii)] if $p=q\in\{1,2\}$, then $a\in \ell_p$;
\item[(iv)] if $1\le q<p\le 2$, then $a\in\ell_q$;
\item[(v)] if $1\le q\le 2<p\le\infty$, then $a\in \ell_s$ for $s=(\frac1p+\frac1q-\frac12)^{-1}$;
\item[(vi)] if $2<q\le p<\infty$, then $a\in\ell_p$;
\item[(vii)] if $2<p<q\le\infty$, then $a\in\ell_p$;
\item[(viii)] if $2\le q\le p=\infty$, then $a\in\ell_\infty$.
\end{itemize}
\end{theorem} 

\section{Proof of Theorem~\ref{t:main2}}\label{s:main2}

By Corollary~\ref{c:Kadets} and Theorem~\ref{t:Garling}, for any numbers $p,q,r\in[1,\infty]$, the coordinatewise multiplication $\ell_p\times \ell_q\to\ell_r$ preserves unconditional convergence if and only if for every $a\in\ell_{r^*}$ the diagonal operator $d_a:\ell_q \to \ell_{p^*}$ is absolutely summing if and only if the following conditions are satisfied:
\begin{itemize}
\item[(a)] if $q^*\le 2$ and $q^*<p^*$, then $\ell_{r^*}\subseteq \ell_{q^*}$;
\item[(b)] if $1<q^*=p^*<2$, then $\ell_{r^*}\subseteq \ell_{q^{*-}}$;
\item[(c)] if $q^*=p^*\in\{1,2\}$, then $\ell_{r^*}\subseteq \ell_{q^*}$;
\item[(d)] if $p^*<q^*\le 2$, then $\ell_{r^*}\subseteq \ell_{p^*}$;
\item[(e)] if $p^*\le 2<q^*$, then $\ell_{r^*}\subseteq \ell_s$ for $s=(\frac1{q^*}+\frac1{p^*}-\frac12)^{-1}$;
\item[(f)] if $2<p^*\le q^*<\infty$, then $\ell_{r^*}\subseteq\ell_{q^*}$;
\item[(g)] if $2<q^*<p^*$, then $\ell_{r^*}\subseteq\ell_{q^*}$;
\item[(h)] if $2\le p^*\le q^*=\infty$, then $\ell_{r^*}\subseteq\ell_\infty$.
\end{itemize}
Now observe that the conditions (a)--(h) are equivalent to the following conditions (a$'$)--(h$'$), respectively:
\begin{itemize}
\item[(a$'$)] if $2\le q$ and $p<q$, then $r\ge q$;
\item[(b$'$)] if $2<p=q<\infty$, then $r>q$;
\item[(c$'$)] if $p=q\in\{2,\infty\}$, then $r\ge q$;
\item[(d$'$)] if $2\le q<p$, then $r\ge p$;
\item[(e$'$)] if $q<2\le p$, then $r^*\le(\frac1{q^*}+\frac1{p^*}-\frac12)^{-1}$;
\item[(f$'$)] if $1<q\le p<2$, then $r\ge q$;
\item[(g$'$)] if $p<q<2$, then $r\ge q$;
\item[(h$'$)] if $1=q\le p\le 2$, then $r\ge q$.
\end{itemize}
The conditions (a$'$), (c$'$), (e$'$), (f$'$), (g$'$), (h$'$) imply the condition
\begin{itemize}
\item[(i$'$)] if $p\le 2$, then $r\ge q$,
\end{itemize}
and the conditions (a$'$)--(e$'$) imply the conditions 
\begin{itemize}
\item[(ii$'$)] if $2<p<q$, then $r\ge q$;
\item[(iii$'$)] if $2<p=q<\infty$, then $r>q$;
\item[(iv$'$)] if $p=q=\infty$, then $r=\infty$;
\item[(v$'$)] if $2\le q<p$, then $r\ge p$;
\item[(vi$'$)] if $q<2<p$, then $\frac1{q^*}+\frac1{p^*}-\frac12\le \frac1{r^*}$, which is equivalent to $\frac12+\frac1r\le \frac1p+\frac1q$.
\end{itemize}
On the other hand, the conditions (i$'$)--(vi$'$) imply the conditions (a$'$)--(h$'$).

It is easy to see that the conjunction of the conditions  (i$'$)--(vi$'$) is equivalent to the disjunction of the conditions (i)--(vi) in Theorem~\ref{t:main2}, which completes the proof of Theorem~\ref{t:main2}. 

\section{Proof of Theorem~\ref{t:main1}}\label{s:main1}

Theorem~\ref{t:main1} follows immediately from Theorem~\ref{t:unconditional} and the next theorem, which is the main result of this section.

\begin{theorem}\label{t:U}  A Banach action $A\times X\to Y$ is unconditional if $A$ is a Hilbert space possessing an orthonormal basis $(e_n)_{n\in\w}$ such that for every $x\in X$ the series $\sum_{n\in\w}e_nx$ is weakly absolutely convergent.
\end{theorem}

\begin{proof} Assume that $A$ is a Hilbert space and $(e_n)_{n\in\w}$ is an orthonormal basis in $A$ such that for every $x\in X$ the series $\sum_{n\in\w}e_nx$ is weakly absolutely convergent  $Y$. For any $y^* \in Y^*$, consider the following two  operators:
$$
\begin{aligned}
&T_1: X \to \ell_1, \quad T_1:x\mapsto (y^*(e_nx))_{n\in\w},\quad\mbox{and}\\
&T_2: \ell_1 \to A,\quad T_2:(s_n)_{n\in\w}\mapsto \sum_{n\in\w} s_n e_n.
\end{aligned}
$$
Both of them are bounded linear operators (for the boundedness of $T_1$ see, for example \cite[Lemma 6.4.1]{KK}). A fundamental theorem of Grothendieck from his famous paper \cite{Grot} (see, for example, \cite[Theorem 4.3.2]{KK} for the standard proof, and \cite[Section III.F]{Woj} for a different approach) says that every bounded linear operator from $\ell_1$ to a Hilbert space is absolutely summing, so in particular $T_2$  is absolutely summing. Then the composition $T_2T_1$ is absolutely summing as well. Let us demonstrate that  $T_2T_1$ is equal to the operator $D_{y^*}:X\to A$ from Theorem \ref{t:dual} (for the Hilbert space $A$ we identify in the standard way $A^*$ with $A$). This will imply that that $D_{y^*}$ is absolutely summing and thus will complete the proof.

Denote by $\langle\cdot\,,\cdot\rangle$ the inner product in the Hilbert space $A$. By the definition, $\langle D_{y^*}x, a \rangle = y^*(ax)$ for all $a \in A$ and $x \in X$. Now, the expansion of $D_{y^*}x$ with respect to the orthonormal basis  $(e_n)_{n\in\w}$ gives us the desired formula
$$
D_{y^*}x = \sum_{n\in\w}\langle D_{y^*}x, e_n \rangle e_n = \sum_{n\in\w}  y^*(e_nx) e_n = T_2(T_1 x).
$$
\end{proof}

\begin{remark} The Banach action $\ell_2\times\IR\to\IR$, $(a,x)\mapsto\sum_{n\in\w}\frac{a(n)x}{n+1}$, preserves the unconditional convergence but for every nonzero $x\in\IR$ the series $\sum_{n\in\w}e_nx=\sum_{n\in\w}\frac{x}{n+1}$ diverges. This example shows that the weak absolute convergence of the series $\sum_{n\in\w}e_nx$ in Theorem~\ref{t:U} is not necessary for the preservation of unconditional convergence by a Banach action $\ell_2\times X\to Y$.
\end{remark}

\section{Acknowledgements}

The authors express their sincere thanks to Alex Ravsky for fruitful discussions that motivated the authors to investigate the problem of preservation of unconditional convergence by Banach actions, to Fedor Nazarov for the suggestion to apply Hadamard matrices in the proof of Theorem~\ref{t:main2}(vi) (which was included in the preceding  \href{https://arxiv.org/abs/2111.14253}{arXiv:2111.14253v1} version of this paper), and to Daniel Pellegrino for suggesting the idea of the proof of Corollary~\ref{c:main} using results of Bennett \cite{Bennett} on Schur multipliers and their relations to absolutely summing operators. After analyzing the approach of Pellegrino, we found an alternative proof of Corollary~\ref{c:main} (and more general Theorem~\ref{t:main2}), based on the results of Garling \cite{Garling}.

The research of the second author was supported by the National Research Foundation of Ukraine funded by Ukrainian State budget as part of the project 2020.02/0096 ``Operators in infinite-dimensional spaces: the interplay between geometry, algebra and topology''.

\end{document}